\documentclass{article}[14]
\usepackage[utf8]{inputenc}
\usepackage[english]{babel}
\usepackage{mathtools}
\usepackage{amsthm}
\usepackage{authblk}
\usepackage[alphabetic, abbrev]{amsrefs}  

\newtheorem{Theorem}{Theorem}[section]

\newtheorem{Lemma}{Lemma}[section]
\newtheorem{Definition}{Definition}[section]
\newtheorem{Conjecture}{Conjecture}[section]
\begin{document}
\title{Paths With Three Blocks In Digraphs }
\author{Maidoun Mortada $^1$, Amine El Sahili $^2$, 
		\\Zahraa Mohsen $^{1,3}$}
	
	\footnotetext[1]{Lebanese University, KALMA Laboratory, Baalbeck.}
	\footnotetext[2]{Lebanese University, KALMA Laboratory, Beirut.}
	\footnotetext[3]{Paris Diderot University, IMJ Laboratory, Paris.}
\date{}
\maketitle

\begin{center}
\begin{abstract}A path with three blocks $P(k,l,r)$ is an oriented path formed by $k$-forward arcs followed by $l$-backward arcs then $r$-forward arcs. We prove that any $(2k+1)$-chromatic digraph contains a path $P(1,k,1)$. However the existence of $P(1,l,1)$ with $l \geq k$ is established in any $(k+4)$-chromatic digraph. In general, we establish a quadratic bound for paths with three blocks.\\

\end{abstract}

\end{center}

\section{Introduction}

The digraphs considered here have no loops or multiple edges. An oriented graph is a digraph in which, for every two vertices $x$ and $y$, at most one of $(x,y)$ or $(y,x)$ is an edge. The chromatic number of a digraph is the chromatic number of its underlying graph. A graph $G$ is said to be $k$-critical if $\chi(G)=k$  and $\chi(G-v)=k-1$ for any vertex $v$ in $V(G)$ .

A block of a path in a digraph is a maximal directed subpath. We recall that the length of a path is the number of its edges.

Given an $n$-chromatic digraph $D$, which oriented paths of length $n-1$ can be found in $D$?

 Havet and Thomassé proved that any $n$-tournament contains any oriented path of length $n-1$ except in three cases: the directed 3-cycle, the regular tournament on five vertices, and the Paley tournament on seven vertices; in these cases $D$ contains no antidirected path of length $n-1$ \cite{Havet}.

In general, when $D$ is an $n$-chromatic digraph, the situation is quite different. Addario.et.al. proved that any  $n$-chromatic digraph $D$, contains any path with two blocks of length $(n-1)$\cite{Havetthomasse}. Moreover, Mortada.et.al proved that any $(n+1)$-chromatic digraph contains a three blocks path of length $n-1$, in which two consecutive blocks are of length 1 each\cite{Mortada}. Regardless Burr's result\cite{Burr}, no more results have been established in this direction. Burr's result states that any $(n-1)^2$-chromatic digraph contains any oriented tree of order $n$. Consequently, we easily deduce that any $(n-1)^2$-chromatic digraph contains any oriented path of order $n$.\\

In this paper we are interested in studying three blocks paths. In the second section, we improve the bound to $\dfrac{3}{4}n^{2}$ for paths with three blocks. In the third section, we prove that any $2k+1$-chromatic digraph contains a path $P(1,k,1)$, for $k$ positive integer. Moreover, we show that this becomes $k+4$ if we searched for the existence of $P(1,l,1)$, with $l\geq k$.\\

\section{Paths with Three Blocks in Digraphs with Large Chromatic Number}

We follow in this section the same reasoning as El Sahili's result.\cite{sahili}\\

Let $f(n)$ be the smallest positive integer such that every $f(n)$- chromatic digraph $D$ contains a $P(k,l,r)$ where $k,l,r$ are positive integers such that $k+l+r=n-1$. It follows from the result of Burr \cite{Burr} that $f(n)\leq (n-1)^{2}$.\\
  Given a path $P$, we denote by $\bar{P}$ the path obtained from $P$ by reversing the directions of all arcs of $P$.

\begin{Lemma}
 Let $G$ be a graph containing no $K_{2n+1}$, $n \geq 2$. Suppose that we can orient $G$ in such a way that each vertex has in-degree at most $n$, then $\chi(G) \leq 2n$.\\
\end{Lemma}

\begin{Definition}
Define the sequence $g(m, i)$ for $m \geq 4$ and $0 \leq i \leq \frac{m}{2} -1$ by:
\begin{align*} &g(m,0)=m \text{ and } g(4,1)=4\\
&g(m,i)=g(m-1,i-1)+2(m-3) \text{ for } m\geq 5 \text{ and } i\geq 1.
\end{align*}

\end{Definition}

\begin{Lemma}

Any $g(m,i)$-chromatic digraph $D$ contains any path $P(k,m-1-k-i,i)$ for $k$ positive integer less than $m-1$.

\end{Lemma}

\begin{proof}

We will proceed by induction on $i$.\\
The case is solved by Addario.et.al \cite{Havet thomasse} for $i=0$. The case is deduced from El Sahili \cite{Sahi}, and see also \cite{Mortada} for $i=1$, $m=4$. For $i \geq 1$, $m \geq 5$, suppose to the contrary that $D$ is a $g(m,i)$-chromatic with no $P(k,m-1-k-i,i)$. Let $H$ be the subdigraph of $D$ induced by the vertices of out-degree in $D$ at least $m-2$. Set $H'$
the subdigraph induced by $V(D-H)$. We claim that $H$ contains no path of the form $P(k,m-1-k-i,i-1)$, since else let $Q=$ $P(k,m-1-k-i,i-1)$ in $H$, and let $v$  be the end vertex of $Q$ which is also the end vertex of the block of length $i-1$. Since $d^{+}(v) \geq m-2$, there exists a vertex $u \in N^{+}(v)-Q$, hence $Q \cup (v,u)$ is a path $P(k,m-1-k-i,i)$ in $H$ and so in $D$, a contradiction. Hence by the induction hypothesis $\chi(H)< g(m-1, i-1)$. Now using Lemma 2.1, we can clearly observe that $\chi(H')\leq 2(m-3)$, since else $H'$ contains a tournament of order $2m-5$ and so contains  $P(k,m-1-k-i,i)$ by \cite{Thomason}, contradiction. Therefore $ \chi(D) \leq \chi(H)+\chi(H')< g(m-1,i-1)+2(m-3)=g(m,i)$, a contradiction.
\end{proof}

In the same way we can prove the following lemma:
\begin{Lemma}

Any $g(m,i)$-chromatic digraph $D$ contains any path $\bar{P} (k,m-1-k-i,i)$ for $k$ positive integer less that $m-1$.

\end{Lemma}

\begin{Lemma}

 For every $m \geq 4$, the sequence $g(m,i)$ is increasing with respect to $i$.
\end{Lemma}
\begin{proof}

 We argue by induction on $i$.\\
 For $m=4$, $g(4,1)=4 \geq g(4,0)$.\\
 For $m\geq 5$, $g(m,1)= g(m-1,0)+2(m-3)= m-1+2(m-3) \geq m = g(m,0)$.\\
  Suppose that $g(m,i) \geq g(m,i-1)$.\\
   Then $g(m,i+1) = g(m-1,i) + 2(m-3) \geq g(m-1,i-1)+2(m-3)=g(m,i)$.\\

\end{proof}

\begin{Theorem}

$f(m)\leq \dfrac{3}{4}m^{2}$, $m\geq4$.
\end{Theorem}

\begin{proof}
Since every $g(m,i)$-chromatic digraph contains any $P(k,m-1-k-i,i)$ and any $\bar{P}$ $(k,m-1-k-i,i)$, then $f(m) \leq g\left(m,\dfrac{m}{2}-1\right)$ for  $m$ even and $f(m) \leq g\left(m,\dfrac{m-3}{2}\right)$ for $m$ odd.\\
For $m$ even,
\begin{align*}
g&\left(m,\dfrac{m}{2}-1\right)\\
 &\quad=g\left(m-1,\left(\dfrac{m}{2}-1\right)-1\right)+2(m-3)\\
	&\quad=g\left(m-2,\left(\dfrac{m}{2}-1\right)-2\right)+2((m-1)-3)+2(m-3)\\               
      &\quad=g\left(m-\left(\dfrac{m}{2}-2\right),1\right)+2(m-3)+\cdots+2\left(m-\left(\dfrac{m}{2}-3\right)-3 \right)\\
      &\quad=g\left(m-\left(\dfrac{m}{2}-1\right),0\right)+2(m-3)+\cdots+2\left(m-\left(\dfrac{m}{2}-2\right)-3\right)\\
     &\quad=g\left(\dfrac{m}{2}+1,0\right)+2\sum_{k=0}^{m/2-2}((m-k)-3)\\ 
      &\quad=\dfrac{m}{2}+1+\dfrac{3m^{2}}{4}-\dfrac{7m}{2}+4\\ 
      &\quad=\dfrac{3}{4}m^{2}-3m+5.
\end{align*}               
For $m$ odd,
\begin{align*}
 g&\left(m,\dfrac{m-3}{2}\right)\\
&\quad=g\left(m-1,\left(\dfrac{m-3}{2}\right)-1\right)+2(m-3)\\
&\quad=g\left(m-2,\left(\dfrac{m-3}{2}\right)-2\right)+2((m-1)-3)+2(m-3)\\ 
&\quad=g\left(m-\left(\dfrac{m-3}{2}-1\right),1\right)+2(m-3)+....+2\left(m-\left(\dfrac{m-5}{2}-1\right)-3\right)\\ 
&\quad=g\left(m-\left(\dfrac{m-3}{2}-1\right)-1,0\right)+2(m-3)+....+2\left(m-\left(\dfrac{m-5}{2}\right)-3\right)\\
&\quad=g\left(\dfrac{m+3}{2},0\right)+2\sum^{m/2-5/2}_{k=0}((m-k)-3)\\ 
&\quad=\dfrac{m+3}{2}+\dfrac{3m^{2}}{4}-4m+\dfrac{21}{4}\\ 
&\quad=\dfrac{3}{4}m^{2}-\dfrac{7m}{2}+\dfrac{27}{4}.
\end{align*}           
                  
Hence, $f(m)\leq \dfrac{3}{4}m^{2}$ for all $m$.\\

\end{proof}

\section{ Paths with Three Blocks, P(1,k,1)}

In this section, we are going to improve the bound we found in section 2 for specific forms of $P(k,l,r)$. 

\begin{Theorem}
Any  $(k+4)$-chromatic digraph D contains a $P(1,l,1)$ for some $l \geq k$.
\end{Theorem}
\begin{proof}

Let $D$ be a $(k+4)$-chromatic digraph. Suppose that $D$ contains no $P(1,l,1)$ for all $l\geq k$. Suppose without loss of generality that $D$ is $k+4$-critical, and so $d(v)\geq k+3$ for all $v \in V(D)$. Note that $D$ contains a $P(k,1)$. Let $P=P(s,1)=x_{1}x_{2}\cdots x_{s+1}y$ be such that $s \geq k$ and $s$ is maximal. Due to the maximality of $s$, all the in-neighbors of $x_{1}$, if exist, belong to $P$. Besides, all the out-neighbors of $x_{1}$ belong to $P$, since else let $z \in N^{+}(x_{1})-P$, so $(x_{1},z) \cup P$ is a $P(1,s,1)$ for $s\geq k$,contradiction.\\
Let $i$ be the minimal integer such that $x_{i} \in N(x_{1})$ and $i \geq3$. Note that $|[x_{i},x_{s+1}]| \geq |N(x_{1})-\{x_{2},y\}| \geq k+1$. If $(x_{1},x_{i}) \in E(D)$ then the path $(x_{1},x_{2}) \cup (x_{1},x_{i}) \cup [x_{i},x_{s+1}]) \cup (y,x_{s+1})$ is $P(1,l,1)$ for $l\geq k$, contradiction. Else if $(x_{i},x_{1}) \in E(D)$ then the path $(x_{i},x_{1}) \cup [x_{i},x_{s+1}] \cup (y,x_{s+1})$ is $P(1,l,1)$ for $l\geq k$, contradiction.\\
\end{proof}
\newpage
\begin{Conjecture}

Any $(k+3)$-chromatic digraph $D$ contains a $P(1,l,1)$ for some $l\geq k$.\\

\end{Conjecture}

The problem of existence of $P(1,k,1)$ in a digraph $D$ is quite different. Define $f(k)$ to be the smallest integer such that any $f(k)$-chromatic digraph contains a $P(1,k,1)$. Based on the existence of paths of the form $P(k,1)$, we may show that $f (k)\leq 3k+4$. Indeed, let $D$ be a $(3k+4)$-chromatic digraph with no $P(1,k,1)$. Let $U$ be the set of all origins of any $P(k,1)$ in $D$. It is clear that $U\neq \emptyset$. If all the vertices in $U$ has an out-degree in $D$ less than $k+1$ then, by lemma 2.1, $\chi(U) \leq 2k+2$. Then $\chi (D-U) \geq k+2$ and so $D-U$ contains a $P(k,1)$, a contradiction. Hence there exists $u \in U$ such that $d^{+}(u) > k+1$. let $P_{u}(k,1)$ be a path of the form $P(k,1)$ in $D$ with origin $u$. Let $z \in N^{+}(u) - P_{u}(k,1)$ and so the path $P_{u}(k,1) \cup (u,z)$ is $P(1,k,1)$, contradiction.\\

In the sequel, we improve the bound above of $f(k)$ to $2k+1$.\\
\begin{Theorem}
$f(k)\leq 2k+1$.
\end{Theorem}
\begin{proof}
Let $D$ be a digraph with $\chi(D)=2k+1$ with $k \geq2 $ positive integer. Suppose without loss of generality that $D$ is critical. We argue by the way of contradiction assuming that $D$ has no $P(1,k,1)$. Let $D_{1}$ be the subdigraph of $D$ induced by the vertices of $D$ with $d_{D}^{-}(v)\geq k+1$ and $D_{2}$ the subdigraph of $D$ induced by $V(D)-V(D_{1})$. Note that all the vertices in $D_{2}$ have an out-degree in $D$ is greater than or equal $k$. We either have $\chi(D_{2})\geq k$ or $\chi(D_{1})\geq k+2$.\\

 Suppose that $\chi(D_{2})\geq k$.  Let $P$ be a directed path of maximal length in $D_{2}$. It's clear that $l(P) \geq k-1$. We claim that $l(P)\geq k$. Indeed, suppose to the contrary that $l(P)=k-1$. Set $P=x_{1}\cdots x_{k}$. Note that $N^{+}_{D_{2}}(x_{k})\subset P$ then $x_{k}$ has at least two out-neighbors in $D_{1}$, say $a_{1}$ and $a_{2}$. Since $d^{+}(x_{1})\geq k$, let $v$ an out-neighbor of $x_{1}$ outside $P$. If $v\notin \{a_{1},a_{2}\}$ then $N^{-}(a_{i})=P\cup \{v\}$ for $i=1,2$, since else let $z\in N^{-}(a_{i})-(P\cup \{v\})$ then $ (z,a_{i}) \cup (x_{k},a_{i}) \cup P \cup (x_{1},v)$ is a $P(1,k,1)$, contradiction. But now $(v,a_{2})   \cup (x_{k},a_{2})\cup P \cup (x_{1},a_{1})$ is a $P(1,k,1)$, contradiction. Thus $v\in \{a_{1},a_{2}\}$. Without loss of generality consider that $v=a_{2}$. Note that $N^{-}(a_{1})=P\cup \{v\}$ and $N^{+}(x_{1})\subseteq P\cup \{v\}$. Since $(a_{2},a_{1})\in D$ and $d^{-}(a_{2})\geq k+1 $, then $a_{2}$ has an in-neighbor outside $P\cup \{a_{1}\}$, say $x$. Set $(x,a_{2}) \cup (x_{k},a_{2})\cup P \cup (x_{1},a_{1})$ is a $P(1,k,1)$, contradiction. \\

Set then $P=x_{1}\cdots x_{r+k}$ with $r\geq1$ and $P_{1}=x_{r}\cdots x_{r+k}$. \\

We now claim that $N^{+}(x_{r+k})\cap D_{1}= \emptyset$. Suppose to the contrary that $N^{+}(x_{r+k})\cap D_{1}\neq \emptyset$. If $|N^{+}(x_{r+k})\cap D_{1}|\geq 2$. Then a contradiction can be reached as before. Then $|N^{+}(x_{r+k})\cap D_{1}|=1$. Let $a_{1} \in N^{+}(x_{r+k})\cap D_{1}$. We claim that $x_{r} \in N^{-}(a_{1})$. Suppose not, then there exists $w\in N^{-}(a_{1})-P_{1}$. If $(x_{r+k-1}, a_{1}) \in E(D)$, then $N^{+}(x_{r}) = [x_{r+1},x_{r+k}]$ and so $(x_{r},x_{r+2}) \in E(D)$. Consequently, $(x_{r},x_{r+1})\cup (x_{r},x_{r+2}) \cup [x_{r+2},x_{r+k}] \cup (x_{r+k}, a_{1}) \cup (w,a_{1})$ is a $P(1,k,1)$, contradiction. Hence $(x_{r+k-1}, a_{1}) \notin E(D)$ and so there exists $w'\in N^{-}(a_{1})-P_{1}-\{w\}$. In this case $N^{+}(x_{r+1}) = [x_{r+2},x_{r+k}] \cup \{a_{1}\}$. And so $N^{+}(x_{r}) = [x_{r+1},x_{r+k}]$ in particular $(x_{r},x_{r+2}) \in E(D)$, contradiction as before. Hence $x_{r} \in N^{-}(a_{1})$. Consequently, $N^{+}(x_{r+1}) = [x_{r+2},x_{r+k}] \cup \{a_{1}\}$. In particular $(x_{r+1},x_{r+k}) \in E(D)$, but $N^{+}(x_{r+k}) \geq k$ then there exists $x_{j} \in N^{+}(x_{r+k})$ for $1 \leq j<r$ so $r\geq2$. As before $x_{r-1} \in N^{-}(a_{1})$. Then $x_{r+k}$ has at least two out-neighbors in $[x_{1},x_{r}[$, let $x_{i}$ be one of them with $i$ maximal, then $(x_{i-1},x_{i}) \cup [x_{r+1}, x_{r+k}] \cup (x_{r+k},x_{i}) \cup (x_{r+1},a_{1})$ is a $P(1,k,1)$, contradiction. \\

Next, if $(x_{r+k},x_{r}) \in E(D)$ then $(x_{r+k},x_{1})\in E(D)$. Suppose to the contrary that $(x_{r+k},x_{1}) \notin D$.
Since $(x_{r+k},x_{r}) \in E(D)$ then $N^{+}(x_{r+1})\subseteq [x_{r+2},x_{r+k}] \cup \{x_{r-1}\}$, since else we get a $P(1,k,1)$, a contradiction. But $d^{+}(x_{r+1}) \geq k$, then $N^{+}(x_{r+1})= [x_{r+2},x_{r+k}] \cup \{x_{r-1}\}$. In particular, $(x_{r+1},x_{r+k}) \in E(D)$. Since $d^{+}(x_{r+k}) \geq k$, then there exists $1<i<r-1$ such that $x_{i} \in N^{+}(x_{r+1}) $. Hence $(x_{i-1},x_{i}) \cup [x_{r+1},x_{r+k}] \cup (x_{r+k},x_{i}) \cup (x_{r+1},x_{r-1})$ is $P(1,k,1)$, a contradiction.\\ 

Now if $|N^{+}(x_{r+k}) \cap [x_{2},x_{r}[|\geq 3$, then we get $P(1,k,1)$, contradiction. If $|N^{+}(x_{r+k}) \cap [x_{2},x_{r}[|=2$ then so $(x_{r+k},x_{1}) \in E(D)$, since else we get that $(x_{r+k},x_{r+1}) \in E(D)$, but since $d^{+}(x_{r+1}) \geq k$ then $x_{r+1}$ has at least two out neighbors outside $P_{1}$ which gives a contradiction as before.
If $|N^{+}(x_{r+k}) \cap [x_{2},x_{r}[|\leq1$, then $(x_{r+k},x_{1}) \in E(D)$.
Hence $x_{i}$ can play the role of $x_{r+k}$ for every $x_{i} \in V(P)$. We claim that $N^{-}(x_{i}) \subset P$ for every $x_{i} \in V(P)$. Otherwise, without loss of generality suppose that there exists $z \in N^{-}(x_{r+k})-P$. $z\in D_{1}$, since else $ (z,x_{r+k})\cup (x_{r+k},x_{1}) \cup [x_{1},x_{r+k-1}] $ is a directed path in $D_{2}$ longer than $P$, a contradiction. Since $N^{+}(x_{r}) \subseteq P$ then $N^{+}(x_{r})=[x_{r+1},x_{r+k}]$, otherwise we get a $P(1,k,1)$, in particular, $(x_{r},x_{r+2})$ and $(x_{r},x_{r+k}) \in E(D)$. But $d^{+}(x_{r+k}) \geq k$, then there exists $x_{j} \in N^{+}(x_{r+k}) $ such that $1<j<r$. Hence $(x_{r},x_{r+1}) \cup (x_{r},x_{r+2}) \cup [x_{r+2},x_{r+k}] \cup (x_{r+k},x_{j}) \cup (x_{j-1},x_{j})$ is a $P(1,k,1)$, a contradiction. Hence $N^{-}(x_{i}) \subset P$ for every $x_{i} \in V(P)$. Since also $N^{+}(x_{i}) \subset P$ and $D$ is critical then $V(D)=V(D_{2})=V(P)$. Since $d^{+}(x_{i}) \geq k$ and $d^{-}(x_{i}) \leq k$ for every $x_{i} \in V(P)$, then $d^{+}(x_{i})=d^{-}(x_{i})=k$ for every $x_{i} \in V(P)$ and so $D$ is a $2k$-regular digraph. Since $\chi(D)=2k+1$, then $D$ is a $2k+1$-tournament and so it contains a $P(1,k,1)$, a contradiction.\\

Consequently, $\chi{(D_{2})}<k$ and so $\chi{(D_{1})} \geq k+2$ , then let $P=z_{1}\cdots z_{k+r}$ be a directed path in $D_{1}$ with maximal length in ($r\geq2$) and let $P_{2}=z_{1}\cdots z_{k+1}$.\\
We prove that $N^{-}(z_{1}) \cap D_{2}= \emptyset $. Suppose to the contrary that $N^{-}(z_{1}) \cap D_{2}\neq \emptyset $. Let $a_{1} \in D_{2} \cap N^{-}(z_{1})$. Since $d^{-}(z_{k}) \geq k+1$, then there exists $v \in N^{-}(z_{k}) -[z_{1},z_{k+1}]- \{ a_{1} \}$ . Note that $d^{+}(a_{1}) \geq k$. If there exists $w \in N^{+}(a_{1})-[z_{1},z_{k}]-\{v\}$, we get a $P(1,k,1)$, a contradiction. Hence $N^{+}(a_{1}) \subseteq ([z_{1},z_{k}]\cup \{v\})$. If $(a_{1},z_{2})\in E(D)$ then $N^{-}(z_{k+1}) = [z_{1},z_{k}] \cup \{a_{1}\}$ since else we get a $P(1,k,1)$, a contradiction. But $N^{+}(a_{1}) \subseteq ([z_{1},z_{k}]\cup \{v\})$, a contradiction. And so $(a_{1},z_{3}) \in E(D)$. Since $d^{-}(z_{k+2})\geq k+1$ and $z_{k+2}\notin N^{+}(a_{1})$, then there exists $x\in N^{-}(z_{k+2})-([z_{3},z_{k}]\cup \{a_{1},z_{1}\})$. Hence $(x,z_{k+2}) \cup [z_{3},z_{k+2}]\cup (a_{1},z_{3}) \cup (a_{1},z_{1})$ is a $P(1,k,1)$, contradiction.\\
Since $d^{-}(z_{1}) \geq k+1$ then there exists $z_{t_{1}},z_{t_{2}} \in N^{-}(z_{1})$ such that $k<t_{1}<t_{2}<k+r$. Also since $d^{-}(z_{k}) \geq k+1$ then $|N^{-}(z_{k})- [z_{1},z_{k-1}]| \geq 2$. Set $I=N^{-}(z_{k})- [z_{1},z_{k-1}]$. $I\subseteq \{z_{t_{1}},z_{t_{1}+1}\} \cap  \{z_{t_{2}},z_{t_{2}+1}\}$ since else we get a $P(1,k,1)$, a contradiction. But $|I| \geq2$, and so $|\{z_{t_{1}},z_{t_{1}+1}\} \cap  \{z_{t_{2}},z_{t_{2}+1}\}|\geq2$ then $\{t_{1},t_{1}+1\}=\{t_{2},t_{2}+1\}$, a contradiction.\\

\end{proof} 
 
\bigskip

\textbf{Acknowledgment}.The authors would like to acknowledge the National Council for Scientific Research of Lebanon (CNRS-L) and the Agence Universitaire de la Francophonie in cooperation with Lebanese University for granting a doctoral fellowship to Zahraa Mohsen.

\bibliographystyle{alphaabbrv}
\bibliography{biblio-v5}


\end{document}